\documentclass[12pt]{amsart}
\pdfoutput=1
\usepackage{amssymb}
\usepackage{amsmath}
\usepackage{amsfonts}
\usepackage[usenames]{color}
\usepackage{graphicx}
\usepackage{array}
\usepackage{psfrag}
\usepackage{color}
\usepackage{ulem}
\usepackage{fancyhdr}
\usepackage[table,xcdraw]{xcolor}
\usepackage{multirow}
\usepackage[pagewise]{lineno}
\usepackage{fancyhdr}

\makeatletter
 
 \@addtoreset{equation}{section}
\makeatother

\newtheorem{theorem}{Theorem}
\newtheorem{lemma}[theorem]{Lemma}
\newtheorem{proposition}[theorem]{Proposition}

\theoremstyle{definition}

\newtheorem{corollary}[theorem]{Corollary}

\theoremstyle{remark}
\newtheorem{remark}[theorem]{Remark}

\numberwithin{equation}{section}

\newcommand{\R}{\mathbb{R}}

\newcommand{\dfn}[1]{\textit{#1}}

\begin{document}

\title[Constructions stemming from non-separating planar graphs ]{Constructions stemming from non-separating planar graphs and their Colin de Verdi\`ere invariant}

\date{\today}
\author{Andrei Pavelescu}
\address{Department of Mathematics and Statistics, University of South Alabama, Mobile, AL 36688}
\email{andreipavelescu@southalabama.edu}

\author{Elena Pavelescu}
\address{Department of Mathematics and Statistics, University of South Alabama, Mobile, AL 36688}
\email{elenapavelescu@southalabama.edu}

\maketitle

\begin{abstract}
A planar graph $G$ is said to be non-separating if there exists an embedding of $G$ in $\mathbb{R}^2$ such that for any cycle $\mathcal{C}\subset G$, all vertices of $G\setminus \mathcal{C}$ are within the same connected component of $\mathbb{R}^2\setminus \mathcal{C}$. 
Dehkordi and Farr classified the non-separating planar graphs as either outerplanar graphs, subgraphs of wheel graphs, or subgraphs of elongated triangular prisms.
We use maximal non-separating planar graphs to construct examples of maximal linkless graphs and maximal knotless graphs.
We show that for a maximal  non-separating planar graph $G$ with $n\ge 7$ vertices, the complement $cG$  is $(n-7)-$apex.
This implies that  the Colin de Verdi\`ere invariant of the complement $cG$ satisfies  $\mu(cG)  \le n-4$.
We show this to be an equality.
As a consequence, the conjecture of Kotlov, Lov\`asz, and Vempala that for a simple graph $G$,  $\mu(G)+\mu(cG)\ge n-2$ is true for 
2-apex graphs $G$ for which $G-\{u,v\}$ is planar non-separating.
It also follows that complements of non-separating planar graphs of order at least nine are intrinsically linked. 
We  prove that the complements of non-separating planar graphs $G$ of order at least ten are intrinsically knotted. 
 \end{abstract}
\vspace{0.1in}

\section{Introduction}

All graphs in this paper are finite and simple.
A graph is \dfn{intrinsically linked} (IL) if every embedding of it
in $\R^3$ (or $S^3$) contains a nontrivial 2-component link.
A graph is \dfn{linklessly embeddable} if it is not intrinsically linked (nIL).
A graph is \dfn{intrinsically knotted} (IK) if every embedding of it in $\R^3$ (or $S^3$) contains a nontrivial knot.
The combined work of Conway and Gordon \cite{CG},
Sachs \cite{Sa}, and Robertson, Seymour and Thomas \cite{RST} fully characterize IL graphs: a graph is IL if and only if it contains a graph in the Petersen family as a minor.
The Petersen family consists of seven graphs obtained from the complete graph $K_6$ by $\nabla Y-$moves and $Y\nabla-$moves, as described in Figure~\ref{fig-ty}. 
The $\nabla Y-$move  and the $Y\nabla-$move preserve the IL property. 
While $K_7$ and $K_{3,3,1,1}$ together with many other minor minimal IK graphs  have been found \cite{GMN}, \cite{CG}, \cite{Foisy}, a characterization of IK graphs is not fully known. 
While the $\nabla Y-$move preserves the IK property \cite{MRS}, the  $Y\nabla-$move doesn't preserve it \cite{FN}.
A graph is said to be $k-apex$ if it can be made planar by removing $k$ vertices. If $G$ and $H$ denote two simple graphs with vertex sets $V(G)$ and $V(H)$, and edge sets $E(G)$ and $E(H)$, respectively, then the \textit{sum} $G+H$ denotes the simple graph with vertex set $V(G)\sqcup V(H)$ and edge set $E(G) \sqcup E(H)\sqcup L$, where $L$ denotes the set of all edges with one endpoint in $V(G)$ and the other in $V(H)$.


\begin{figure}[htpb!]
\begin{center}
\begin{picture}(160, 42)
\put(0,0){\includegraphics[width=2.4in]{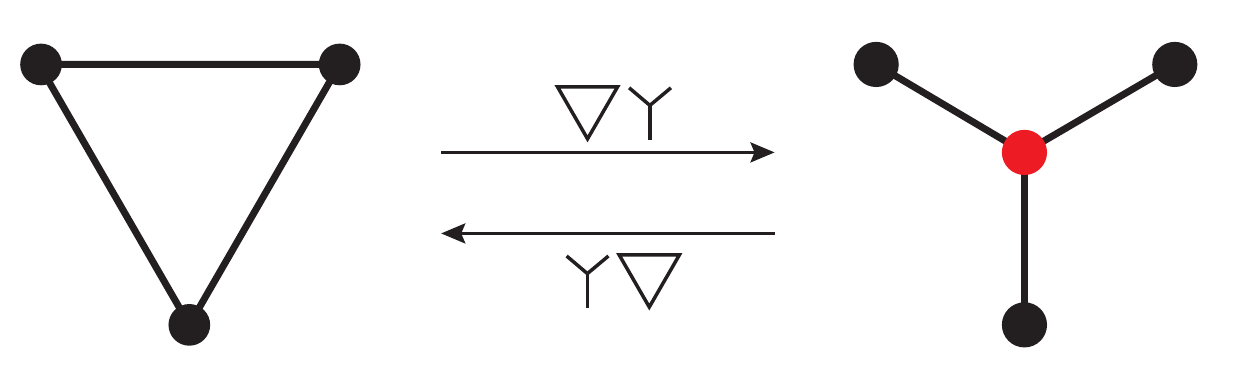}}
\end{picture}
\caption{\small  $\nabla Y-$ and $Y\nabla-$moves} 
\label{fig-ty}
\end{center}
\end{figure}

 A planar graph $G$ is \textit{non-separating} if there exists an embedding of $G$ in $\mathbb{R}^2$ such that for any cycle $\mathcal{C}\subset G$, all vertices of $G\setminus \mathcal{C}$ are within the same connected component of $\mathbb{R}^2\setminus \mathcal{C}$.
By work of Dehkordi and Farr \cite{DF}, a non-separating planar graph is one of three types: (1) an outerplanar graph, (2) a subgraph of a wheel, (3) a subgraph of an elongated triangular prism.  
In Section \ref{max}, we consider sums between maximal non-separating planar graphs and small empty graphs, complete graphs or paths to construct maximal linklessly embedable graphs and maximal knotlessly embeddable graphs. 
A simple graph $G$ is called \textit{maximal linklessly embeddable} (maxnIL) if it is not a proper subgraph of a nIL graph of the same order.
 A simple graph $G$ is called \textit{maximal knotlessly embeddable} (maxnIK) if it is not a proper subgraph of a nIK graph of the same order.
Constructions and properties of maxnIL graphs can also be found in \cite{A} and \cite{NPP}, and for maxnIK graphs in \cite{EFM}.

Colin de Verdi\`ere \cite{dV}  introduced the graph invariant $\mu$ which is based on spectral properties of matrices associated with  the graph $G$. 
He showed that $\mu$  is monotone under taking minors
and that planarity is characterized by the inequality $\mu\le 3$.  
By work of Lov\'asz and Schrijver  \cite{LS} and Robertson, Seymour, and Thomas \cite{RST}, it is known that linkless embeddability is characterized by the inequality $\mu\le 4$.
By reformulating the definition of $\mu$ in terms of vector labelings, Kotlov, Lov\'asz, and Vempala \cite{KLV} related the topological properties of a graph to the $\mu$ invariant of its complement: 
for $G$ a simple graph on $n$ vertices 
(a) if $G$ is planar, then $\mu(cG)\ge n-5$;
(b) if $G$ is outerplanar, then $\mu(cG)\ge n-4$;
(c) if $G$ is a disjoint union of paths then $\mu(cG)\ge n-3.$
For $G$ a graph with $n$ vertices $v_1, v_2, \ldots v_n$,  $cG$ denotes the \textit{complement of $G$} in the complete graph $K_n$. 
The graph $cG$ has the same set of vertices as $G$ and $E(cG) = \{ v_iv_j | v_iv_j\notin E(G)\}$.

By \cite{BHK}, the complement of a planar graphs with nine vertices is not planar.
This is also implied by the inequality $\mu(cG)\ge n-5$. 
Here we show a stronger inequality  for maximal non-separating planar graphs. 
In Section \ref{mu}, we prove two theorems.

\begin{theorem}
If $G$ is a maximal non-separating planar graph with $n\ge 7$ vertices, then $cG$ is $(n-7)-$apex.
\label{n7ap}
\end{theorem}

\noindent Theorem 1 establishes the upper bound $\mu(cG)\le n-4$  for $G$ a maximal non-separating planar graph, since $\mu\le 3$ for planar graphs and adding one vertex increases the value of $\mu$ by at most one \cite{HLS}.
We prove this  is an equality.

\begin{theorem}
For $G$ a maximal non-separating planar graph with $n\ge 7$ vertices, $\mu(cG)=n-4$.
\label{thmn-4}
\end{theorem}

\noindent  In \cite{KLV}, Kotlov, Lov\`asz, and Vempala conjectured that, for a simple graph $G$,  $\mu(G)+\mu(cG)\ge n-2$. 
We revisit results about $\mu$ to show  the conjecture is true for planar graphs and 1-apex graphs.
As a consequence of Theorem 2, the conjecture holds for 2-apex graphs $G$ for which $G-\{u,v\}$ is planar non-separating.
Theorem 2  also implies that  for $G$ a maximal non-separating planar graph with nine vertices, $\mu(cG)=5>4$, and thus $cG$ is intrinsically linked.
While the relationship between the $\mu$ invariant and intrinsic linkness is well understood, the same is not true for intrinsic knottedness. 
The inequality $\mu(cG)\ge n-5$ for planar graphs $G$, implies that complements of planar graphs with ten vertices are intrinsically linked.
Theorem 2 establishes that for $G$ a maximal non-separating planar graph with ten vertices, $\mu(cG)=6$, but this does not imply that $cG$ is intrinsically knotted.
There are known IK graphs with $\mu=5$ \cite{F}, \cite{MNPP}, as well as nIK graphs with $\mu=6$ \cite{FN}.
 In Sections 4 we do a case by case analysis to prove the following theorem.
\begin{theorem}
If $G$  is  a non-separating planar graph on 10 vertices, then $cG$ is intrinsically knotted.
\label{thm10}
\end{theorem}

Since the complement of a non-separating planar graph contains the complement of a maximal non-separating planar graph of the same order as a subgraph,  it suffices to prove Theorem \ref{thm10}  for maximal non-separating planar graphs: (1) maximal outerplanar graphs, (2) the  wheel graph, (3) elongated triangular prisms.

A similar approach to that presented in Section \ref{ten} works to prove that (a) if $G$  is a non-separating planar graph on 7 vertices, then $cG$ is not outerplanar; (b) if  $G$  is  a non-separating planar graph on 8 vertices, then $cG$ is non-planar; (c)  if $G$  is a non-separating planar graph on 9 vertices, then $cG$ is intrinsically linked.
For outerplanar graphs $G$ with at most nine vertices, these results can also be obtained  using the graph invariant $\mu$, since for such  graphs $G$, $\mu(G) \ge n-4$ \cite{LS}.

\section{MaxnIL and maxnIK graphs}
\label{max}

In this section, we use maximal non-separating planar graphs to build examples of maxnIL and maxnIK graphs.
J{\o}rgensen \cite{J}, and Dekhordi and Farr \cite{DF} considered the class of graphs of the type $H+E_2$ where $E_2$ denotes the graph with two vertices and no edges, and $H$ is an elongated prism. J{\o}rgensen proved that these graphs are maximal with no $K_6$ minors. 
Dekhordi and Farr proved that these graphs are maxnIL. 
Here we add to this type of examples by taking the sum of maximal non-separating planar graphs with small empty graphs, complete graphs, and paths.
We  use results of Sachs' \cite{Sa} saying that 1-apex graphs are nIL and 2-apex graphs are nIK.
In several cases, we prove maximality with Mader's theorem  \cite{Ma} on existence of complete minors: a graph $G$ with $n$ vertices  and $4n-9$ edges, $n\ge 6$,  contains a $K_6$ minor;  a graph $G$ with $n$ vertices  and $5n-14$ edges, $n\ge 7$,  contains a $K_7$ minor.

 A vertex of a graph $H$  which is incident to all the other vertices of $H$ is a \textit{cone}. We also say that $v$ cones over the subgraph induces by all the vertices of $H$ minus $v$. Let $W_n$ denote the wheel graph of order $n\ge 4$. 

\begin{theorem} 
 The graph $G\simeq W_n+E_2$ is maxnIL.
\end{theorem}

\begin{proof}  Let $\{v_1, v_2,....,v_n\}$ denote the vertices of $W_n$, and assume that $v_n$ is adjacent to $v_i$, for $1\le i \le n-1$.
Removing the vertex $v_n$ from $G$  yields the subgraph $C_{n-1}+E_2$ which admits a planar embedding, thus $G$ is an 1-apex graph. By \cite{Sa}, $G$ is nIL. Maximality follows by Mader \cite{Ma}, since $G$ has $4|V(G)|-10$ edges. Any added edge creates a $K_6$ minor.

\end{proof}

Let $P_2$ be the graph with vertex set $V(P_2)=\{u,v,w\}$, and edge set $E(P_2)=\{\{u,w\},\{v,w\}\}$. Let $K_3$ denote the complete graph on vertices $\{u,v,w\}$.

\begin{theorem} 
 The graph  $G\simeq W_n+P_2$ is maxnIK.
\end{theorem}

\begin{proof} Let $\{v_1, v_2,....,v_n\}$ denote the vertices of $W_n$, and assume that $v_n$ is adjacent to $v_i$, for $1\le i \le n-1$. 
By removing the vertices $v_n$ and $w$ from $G$, one obtains the planar graph $C_{n-1}+E_2$.
Thus $G$ is a 2-apex graph, hence nIK by \cite{Sa}. 
Maximality follows by Mader \cite{Ma}, since $G$ has $5|V(G)|-15$ edges and any added edge creates a $K_7$ minor.
\end{proof}

\begin{theorem}  If $H$ is a maximal outerplanar graph of order $n\ge 4$, then $G\simeq H+K_2$ is a maxnil graph.
\end{theorem}
\begin{proof}
Since $H$ is outerplanar, $H+K_1$ is planar and $H+K_2$ is 1-apex and therefore nIL \cite{Sa}.
As $H+K_2$ has $|V(H+K_2)| - 10$ edges, it is maxnil \cite{Ma}.
\end{proof}

\begin{theorem} 
 If H is a maximal outerplanar graph of order $n\ge 4$, then $G\simeq H+K_3$ is a maxnIK graph.
\end{theorem}

\begin{proof} The graph $G$ is 2-apex, and it is therefore nIK, by \cite{Sa}. As any maximal outerplanar graph of order $n$ has exactly $2n-3$ edges, the graph $G$ has $2(|V(G)|-3)-3+3(|V(G)|-3)+3=5|V(G)|-15$ edges and is  maxnIK \cite{Ma}.
\end{proof}

For the elongated prism case, we distinguish two cases, according to the number of non-triangular edges of the triangular prism which are subdivided.

\begin{theorem} 
 Let H denote an elongated prism of order $n\ge 8$ obtained by repeated subdivisions of at most two  of three non-triangular edges of the prism graph. Then $G\simeq H+P_2$ is a maxnIK graph.
\end{theorem}

\begin{proof} 

Assume that $H$ is isomorphic to the graph depicted in Figure \ref{2edgeprism}(a), such that the edge $\{v_3,v_4\}$ is not subdivided. Perform a $\nabla Y-$move on the triangle induced by the vertices $\{v_3,v_4,u\}$, by deleting the edges $\{v_3,v_4\}$, $\{v_3,u\}$, and $\{v_4,u\}$, and adding a new vertex $t$ incident to all of  $\{v_3,v_4,u\}$, to obtain a new graph $G'$. This graph is 2-apex, since deleting the vertices $t$ and $w$ gives the planar graph of Figure \ref{2edgeprism}(b). Thus, $G'$ is nIK, and so must be $G$, as the $\nabla Y-$ move preserves the IK property \cite{MRS}.

\begin{figure}[htpb!]
\begin{center}
\begin{picture}(450, 130)
\put(-10,0){\includegraphics[width=6in]{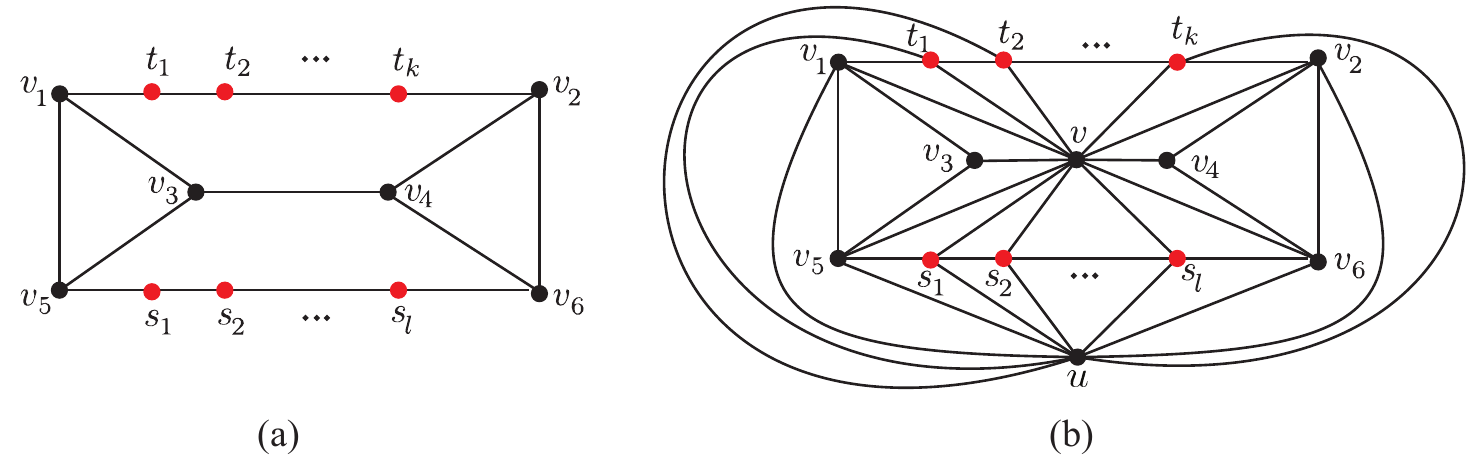}}
\end{picture}
\caption{\small (a) elongated prism with only two edges subdivided; (b) planar graph obtained by deleting the vertices $t$ and $w$ of $H+P_2$ }\label{2edgeprism}
\end{center}
\end{figure}

To show that $G$ is maximal nIK, one  notices  that $G$ is isomorphic to a cone $w$ over $H+E_2$. Since $H+E_2$ is maxnIL by \cite{DF}, adding any edge to $G$ produces a structure of a cone over an IL graph. This structure will contain a minor isomorphic to a graph in either the $K_7$-family or the $K_{3,3,1,1}$-family, and will therefore be IK.

 \end{proof}

\begin{theorem} 
 Let H denote an elongated prism of order $n\ge 9$ obtained by repeated subdivisions of all three non-triangular edges of the prism graph. Then $G\simeq H+P_2$ is an IK graph.
\end{theorem}

\begin{proof} By repeated edge contractions applied to $G$, one obtains the minor $S\simeq P'+P_2$, where $P'$ is the graph depicted in Figure \ref{3prism}(a).

\begin{figure}[htpb!]
\begin{center}
\begin{picture}(400, 80)
\put(30,0){\includegraphics[width=5in]{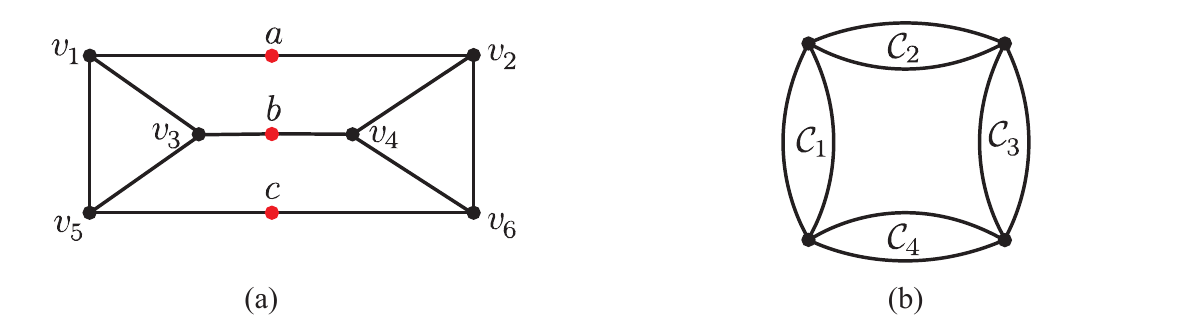}}
\end{picture}
\caption{\small (a) the graph $P'$ obtained by subdividing once each non-triangular edge of the prism graph; (b) the graph $D_4$}
\label{3prism}
\end{center}
\end{figure}

 In \cite{Foisy}, Foisy proved that if a graph contains a doubly linked $D_4$-minor in every embedding, the graph must be IK. 
 The graph $D_4$ is depicted in Figure \ref{3prism}(b).  
 An embedding of the graph $D_4$ is doubly linked if the linking numbers $lk(C_1,C_3)$ and $lk(C_2,C_4)$ are both nonzero.
We used a Mathematica program written by Naimi \cite{Na} to show that $S$ has a doubly linked $D_4$-minor in every embedding.

 \end{proof}


\section{The $\mu$ invariant}
\label{mu}

In this section we determine the value of the $\mu$ invariant for complements of maximal non-separating planar graphs. 
By \cite{HLS}, if $G$ is planar with $n$ vertices, then $\mu(cG)\ge n-5$. 
We first show the inequality $\mu(cG)\le n-4$ for graphs $G$ which are maximal non-separating planar.
In Theorem 2, we show this is in fact an equality.

In \cite{KLV}, Kotlov, Lov\`asz, and Vempala conjectured that, for a simple graph $G$, 
$\mu(G)+\mu(cG)\ge n-2$. 
We review that the conjecture holds for planar graphs and 1-apex graphs.
We show that as a consequence of Theorem 2, the conjecture holds for 2-apex graphs $G$ for which $G-\{u,v\}$ is planar non-separating.\\

\noindent \textbf{Theorem 1.}
\textit{If $G$ is a maximal non-separating planar graph with $n\ge 7$ vertices, then $cG$ is $(n-7)-$apex.}
\begin{proof}

{\textit{Outerplanar case}}. Any maximal outerplanar graph $H$ of order $n\ge3$ can be represented by a triangulated $n$ cycle in the plane (with the unbounded face containing all vertices).
The $n-$cycle contains at least one 2-chord, an edge which  forms a triangle with two adjacent edges along the cycle. 
We say that the 2-chord isolates the vertex which is part of the triangle but is not incident to the 2-chord.  
For example, in Figure \ref{fig-wheeln}(a), the 2-chord $v_1v_6$ isolates the vertex $v_7$ and the 2-chord $v_1v_5$ of $H-v_7$ isolates $v_6$.
The complement of the unique maximal outerplanar graph with 5 vertices is $P_3$, a path with three edges, together with an isolated vertex.
It follows that the complement of any maximal outerplanar graph with 7 vertices is planar, since the deletion of two vertices gives a path with three edges and an isolated vertex. 
For example, after the deleting the vertices $v_7$ and $v_6$, the complement of the graph in Figure \ref{fig-wheeln}(a) is the path $v_1v_3v_5v_2$ together with the isolated vertex $v_4$.
Starting with a maximal outerplanar graph with $n\ge 7$ vertices, one can recursively delete $n-7$ isolated  vertices and obtain a maximal outerplanar graph of order 7.
The same sequence of $n-7$ vertex deletions gives a planar subgraph of $cG$.  Thus $cG$ is $(n-7)$-apex.

\noindent {\textit{Wheel case}}. Let $G$ be the wheel on $n$ vertices. 
Then $cG \simeq (K_{n-1}\setminus C_{n-1}) \cup K_1$. 
Let $\{v_1, v_2, \ldots, v_{n-1}\}$ be the vertices of $C_{n-1}$ in consecutive order, as in Figure \ref{fig-wheeln}(b).
Then $cG\setminus \{v_7, v_8, \ldots, v_{n-1}\}$ is a planar graph (the triangular prism added one edge, together with an isolated vertex) and thus  $cG$ is $(n-7)-$apex.
See Figure \ref{fig-wheeln}(c).

\begin{figure}[htpb!]
\begin{center}
\begin{picture}(450, 115)
\put(0,0){\includegraphics[width=6in]{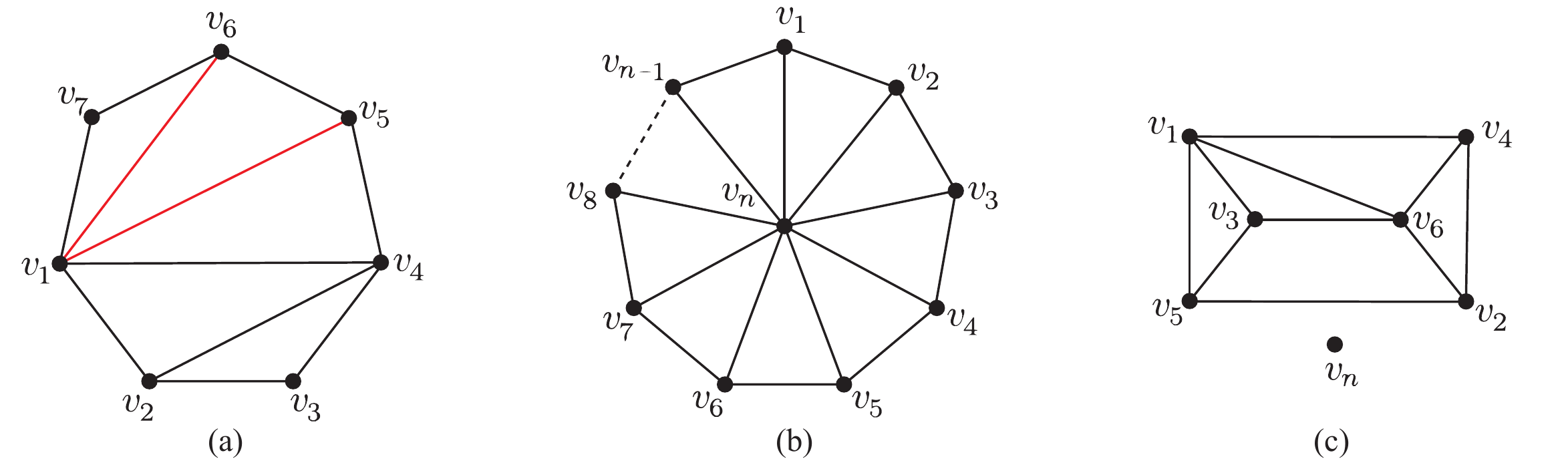}}
\end{picture}
\caption{\small (a) maximal outerplanar graph with 7 vertices; (b) graph $G$, a wheel with  $n$  vertices; (c) $cG\setminus \{v_7, v_8, \ldots, v_{n-1}\}$}
\label{fig-wheeln}
\end{center}
\end{figure}

\noindent {\textit{Elongated prism case}}. 
Let $G$ be an elongated prism with $n\ge 7$ vertices. 
Without loss of generality let $v_1v_3v_5$ be one of two induced triangles of $G$.
Let $a,b$ and $c$ denote their respective neighbors in $V(G)\setminus\{v_1, v_3, v_5\}$ as in Figure \ref{fig-eprism}(a).
Deleting all vertices but $\{v_1, v_3, v_5, a, b, c \}$  in $cG$ gives a subgraph of the outerplanar graph with six vertices in Figure \ref{fig-eprism}(b).
Deleting any $n-7$ vertices of $cG$ none of which is in the set $\{v_1, v_3, v_5, a, b, c\}$ yields a planar graph, thus $cG$ is $(n-7)$-apex.

\begin{figure}[htpb!]
\begin{center}
\begin{picture}(330, 90)
\put(0,0){\includegraphics[width=4.5in]{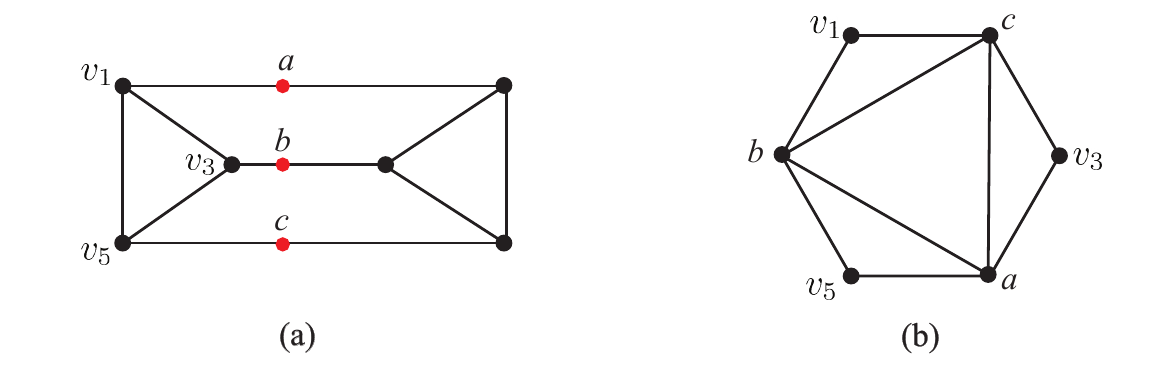}}
\end{picture}
\caption{\small (a) elongated prism (b) subgraph induced by $\{v_1, v_3, v_5, a, b, c\}$ in $cG$ }
\label{fig-eprism}
\end{center}
\end{figure}
\end{proof}

\begin{corollary}
For $G$ a maximal non-separating planar graph with $n\ge 7$ vertices,  $\mu(cG)\le n-4$.
\label{corn-7}
\end{corollary}
\begin{proof}
By Theorem \ref{n7ap}, $cG$ is $(n-7)$-apex. 
Let $H$ be the planar subgraph of $cG$ obtained by deleting $n-7$ vertices.
Then $\mu(H)\le 3$ and $\mu(cG)\le 3 + (n-7) = n-4$, since adding one vertex to a graph increases the value of $\mu$ by at most one (see Theorem 2.7 in \cite{HLS}).
\end{proof}

\noindent Corollary \ref{corn-7} establishes an upper bound of $n-4$ for the values of $\mu$ of complements of maximal non-separating planar graphs on $n$ vertices. 
We show that  $n-4$ is the actual value of $\mu$. 
We use Theorem 5.5 in \cite{HLS}, which says that for $H$ a graph on $n$ vertices and $\nu(H):= n-\mu(cH)-1$, the inequality $\nu(H)\le 2$ holds if and only if $H$ does not
contain as a subgraph any of the five graphs in Figure \ref{fig-nu}. We also use that  for a graph $G$ with at least one edge $\mu(G + K_1) =\mu(G)+1$, by Theorem 2.7 in \cite{HLS}.

\begin{figure}[htpb!]
\begin{center}
\begin{picture}(400, 100)
\put(-20,0){\includegraphics[width=6in]{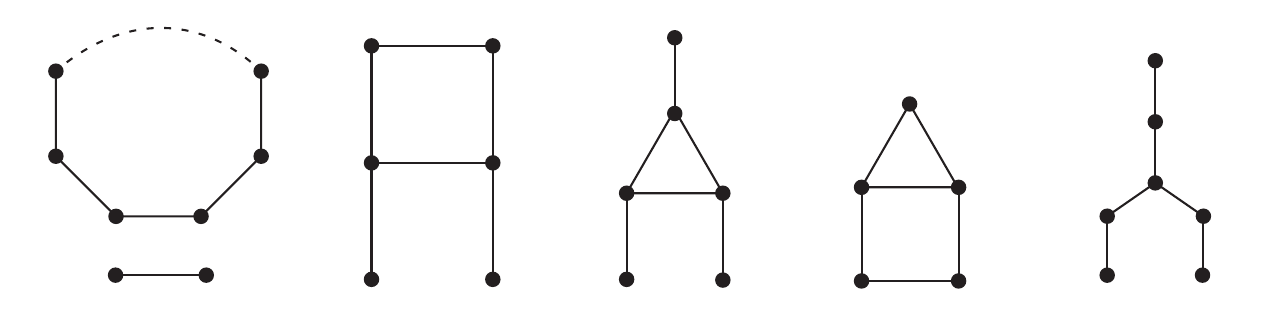}}
\end{picture}
\caption{\small Five graphs}
\label{fig-nu}
\end{center}
\end{figure}

\noindent \textbf{Theorem 2.}
\textit{For $G$ a maximal non-separating planar graph with $n\ge 7$ vertices, $\mu(cG)=n-4$.}
\begin{proof}
Corollary \ref{corn-7} established the inequality $\mu(cG)\le n-4$. 
Here we show that $\mu(cG)\ge n-4$. 
If  $G$ is outerplanar, then $\mu(cG)\ge n-4$  \cite{KLV}.
If $G$ is the wheel graph on $n$ vertices, $cG= cC_{n-1}\cup K_1$.
By Theorem 5.5 in  \cite{HLS},  $\nu(C_{n-1})\le 2$ and we have
$$\mu(cG)=\mu(cC_{n-1})=n-1-\nu(C_{n-1})-1 \ge n-4.$$
For elongated prisms, we distinguish two cases, according to the number of non-triangular edges of the prism which are being subdivided.\\\\
\textbf{Case 1}. Consider $G$ the elongated prism in Figure \ref{fig-eprism1}(a), with exactly one non-triangular edge of the prism graph subdivided, $v_1v_2$,
If at least two vertices are added along $v_1v_2$,  as in Figure \ref{fig-eprism1}(a), consider the graph $H=G- \{v_1,v_2\}$. 
Then $\mu(cH) = (n-2) -\nu(H)-1 \ge n-5$, by Theorem 5.5 in \cite{HLS}.
Since in $cG$, the set of adjacent vertices $\{v_1,v_2\}$ cones over $cH$, $\mu(cG)\ge n-4$, by Theorem 2.7 in \cite{HLS}.
If only one vertex is added along the one edge, as in Figure \ref{fig-eprism1}(b), then $cG$ contains a $K_4$ minor and thus $\mu(cG)\ge 3$. 
See Figure \ref{fig-eprism1}(c).

\begin{figure}[htpb!]
\begin{center}
\begin{picture}(430, 100)
\put(-20,0){\includegraphics[width=6.5in]{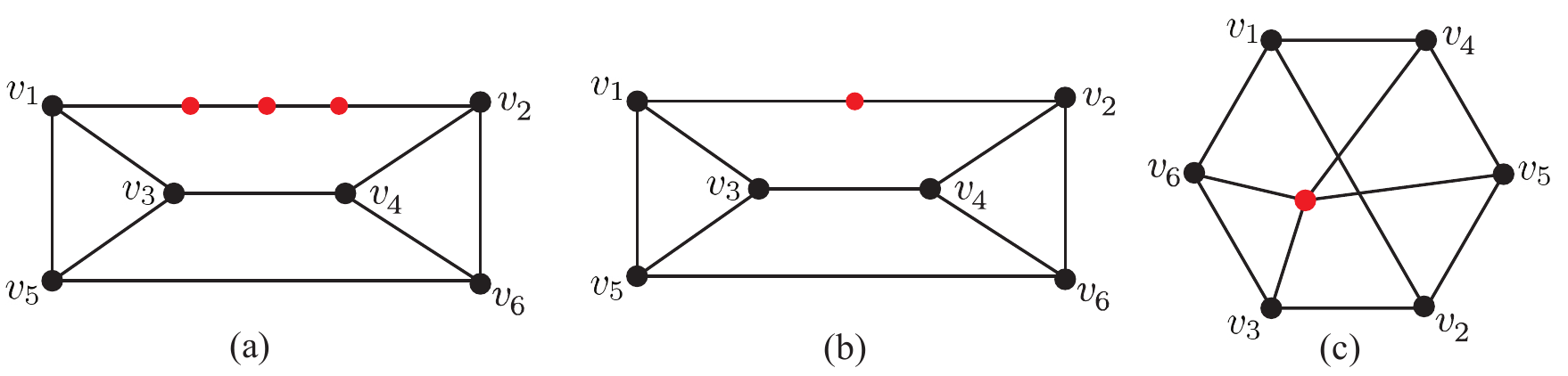}}
\end{picture}
\caption{\small (a)  elongated prism with one non-triangular edge subdivided by more than one vertex;  (b)  elongated prism with one non-triangular edge subdivided by exactly one vertex; (c) complement of graph in (b) }
\label{fig-eprism1}
\end{center}
\end{figure}

\noindent\textbf{Case 2}. 
Assume $G$ is obtained from the triangular prism by subdividing edges $v_1v_2$ and $v_5v_6$ along the way, as in Figure  \ref{fig-eprism2}(a).
The graph $H =G-\{v_1, v_6\}$ is a path with  $n-2$ vertices, so $\mu(cH)\ge n-5$ \cite{KLV}.
In $cG$, the set of adjacent  vertices $\{v_1, v_6\}$ cones over $cH$, yielding $\mu(cG)\ge \mu(cH)+1 \ge n-4$, by Theorem 2.7 in \cite{HLS}.
\end{proof}

\begin{figure}[htpb!]
\begin{center}
\begin{picture}(430, 80)
\put(-20,0){\includegraphics[width=6.5in]{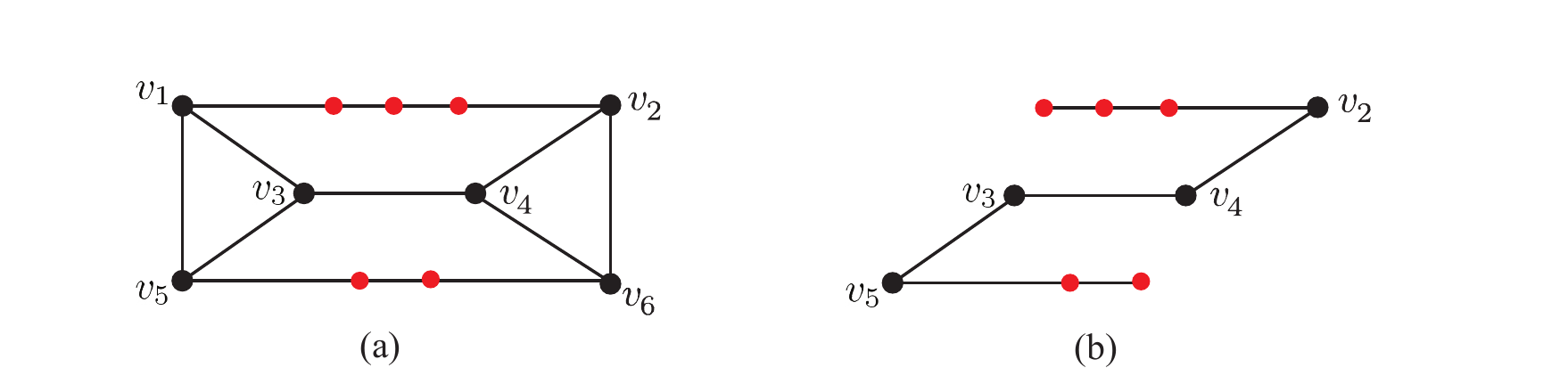}}
\end{picture}
\caption{\small (a) elongated prism  $G$ with two subdivided edges  (b) $H =G-\{v_1, v_6\}$ }
\label{fig-eprism2}
\end{center}
\end{figure}

\noindent We briefly discuss the state of a conjecture of Kotlov, Lov\`asz, and Vempala  \cite{KLV} that for a simple graph $G$ on $n$ vertices, 
$\mu(G)+\mu(cG)\ge n-2$. By work in \cite{KLV}, \cite{dV}, and \cite{HLS}, the conjecture holds if either one of $G$ or $cG$ is planar.
We note that the conjecture holds if $\mu(G)\ge n-6$ or $\mu(cG)\ge n-6$. Assume $\mu(G)\ge n-6$: if $\mu(cG)\ge 4$, then $\mu(G) +\mu(cG)\ge n-2$; if $\mu(cG)<4$, $\mu(G)$ is planar, and the conjecture holds. 
\begin{proposition} The conjecture holds for 1-apex graphs.
\end{proposition}
\begin{proof}
Let $G$ be an 1-apex graphs with $n$ vertices and $H=G-\{v\}$ planar. 
Then $\mu(cH)\ge (n-1)-5 = n-6$ \cite{KLV}. 
We have that $cH$, the complement of $H$ in $K_{n-1}$, is a subgraph of $cG$, the complement of $G$ in $K_n$, since $cG$ may have additional edges incident to $v$, and so $n-6 \le \mu (cH) \le  \mu(cG)$. Thus the conjecture holds for $G$.
\end{proof} 

\begin{corollary}
Let $G$ be a 2-apex graphs with $n$ vertices with $H=G-\{u,v\}$ planar non-separating. Then $\mu(G)+\mu(cG)\ge n-2$.
\end{corollary}
\begin{proof}
Since $H$ is planar, non-separating, by Theorem \ref{thmn-4}, $\mu(cH)\ge (n-2)-4 =n-6$, with equality if $H$ is maximal. 
We have that $cH$, the complement of $H$ in $K_{n-2}$, is a subgraph of $cG$, the complement of $G$ in $K_n$, since $cG$ may have additional edges incident to $u$ and $v$,  and so $\mu(cG)\ge \mu(cH)\ge n-6$. Thus the conjecture holds for $G$.

\end{proof}

\section{Graphs of order ten}
\label{ten}

The relationship between the $\mu$ invariant and the property of being intrinsic knotted is not well understood. 
While, Theorem \ref{thmn-4} establishes that for $G$ a maximal non-separating planar graph with ten vertices, $\mu(cG)=6$,  this information has no bearing on whether $cG$ is intrinsically knotted.
In \cite{FN}, Flapan and Naimi prove that the IK property is not preserved by the $Y\nabla$ move by showing  a graph in the $K_7$ family which is not intrinsically knotted. 
Since $\mu(K_7)=6$ and both the $\nabla Y$ move and the $Y\nabla$ move preserve $\mu$ for $\mu\ge 4$ \cite{HLS},  this nIK graph has $\mu=6$.
On the other hand,  \cite{F} and  \cite{MNPP} provide examples of IK graphs with $\mu=5$.
 In this section, we do a case by case analysis to prove that  for $G$ a maximal non-separating planar graph with ten vertices, $cG$ is intrinsically knotted.
 We recall  that the $\nabla Y$ move preserves the IK property. 
 In some cases, graphs are shown to be IK because they are obtained through one or more $\nabla Y$ moves from IK graphs such as $K_7$ or $K_{3,3,1,1}$.
 In other cases, graphs $G$ are shown to be IK  because  the graphs obtained from $G$ by one or more $Y\nabla$ moves contain $K_7$ or $K_{3,3,1,1}$ minors.
\begin{lemma} If $G$ is a maximal outerplanar graph with ten vertices then $cG$ is intrinsically knotted.
\label{out10}
\end{lemma}
\begin{proof}
We label the vertices of $G$ by $v_1, v_2, \ldots, v_9, v_{10}$ in clockwise order around the cycle $\mathcal{C}$ bordering the outer face of a planar embedding. See Figure \ref{fig-10outerplanar1}. 
We organize the proof according to the longest chord of $\mathcal{C}$. 
The length of a chord is defined as the length of the shortest path in $\mathcal{C}$ between the endpoints of the chord.
In each case we show the complement $cG$ contains an intrinsically knotted graph as a minor.
We remark that within any triangulation of the disk bounded by $\mathcal{C}$,  out of a total of seven chords, at most six have length 2 or 3. 
Thus there exists chords of length 4 or 5.

\begin{figure}[htpb!]
\begin{center}
\begin{picture}(530, 115)
\put(0,0){\includegraphics[width=6.6in]{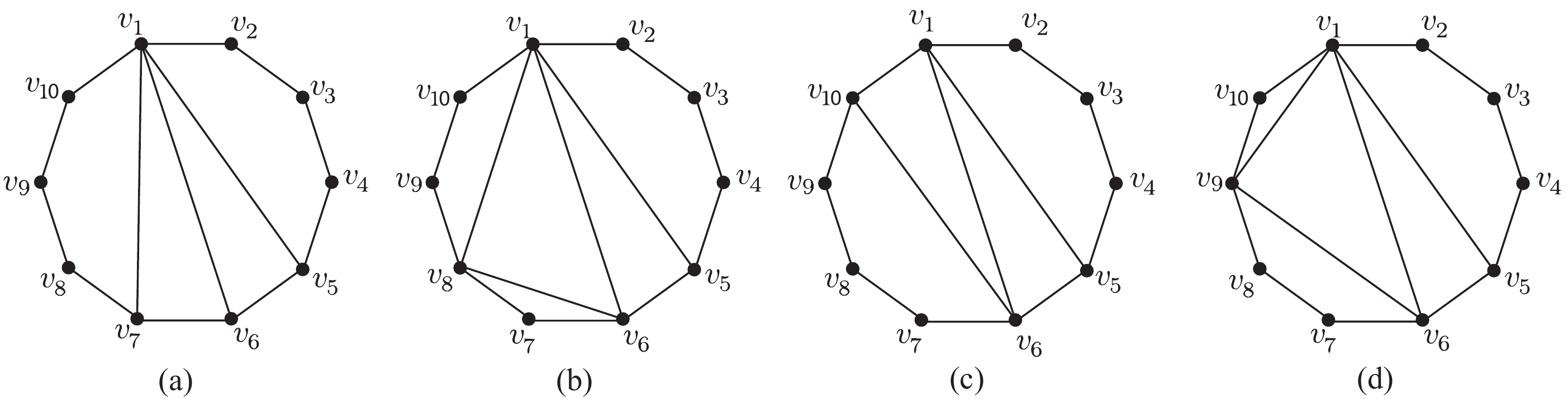}}
\label{fig-10outerplanar1}
\end{picture}
\caption{\small Outerplanar graphs with 10 vertices.}
\label{fig-10outerplanar1}
\end{center}
\end{figure}

\noindent {\textbf{Case (a)}}  If the cycle $\mathcal{C}$ has a chord of length 5, we may assume without loss of generality that $v_1v_6\in E(G)$. 
Consider the cycles $\mathcal{C}_1:=v_1v_6v_7v_8v_9v_{10}$ and $\mathcal{C}_2:=v_1v_2v_3v_4v_5v_6$.
We note that $\mathcal{C}$ necessarily contain a 3-chord or a 4-chord with one endpoint at $v_1$ or $v_6$ and the other endpoint among the vertices of $\mathcal{C}_i$, for $i=1,2$.
We distinguish six cases, according to whether there are any 4-chords at all and whether these chords share one of their ends.

\noindent {\bf{(a1)}}  Assume there exists a 4-chord incident to $v_1$ or $v_6$, say $v_1v_5\in E(G)$.
\begin{enumerate}
\item[(i)] If $v_1v_7\in E(G)$ (see Figure \ref{fig-10outerplanar1}(a)), then the complement $cG$ contains as a subgraph  the graph obtained through two $\nabla Y$-moves from $K_7$ with vertex set $\{v_2, v_3, v_4, v_8, v_9,  v_{10}, v_6\}$: one  $\nabla Y$-move over the triangle $v_2v_3v_4$ with new vertex $v_7$ and one $\nabla Y$-move over the triangle $v_8v_9v_{10}$ with new vertex $v_5$. 
\item[(ii)] If  $v_1v_7\notin E(G)$, and $v_1v_8\in E(G)$ (see Figure \ref{fig-10outerplanar1}(b)), then in  $cG$ delete any edges incident to $v_5$ except $v_5v_8$, $v_5v_9$ and $v_5v_{10}$, then perform a $Y\nabla$-move at $v_5$ to create a graph containing the triangle $v_8v_9v_{10}$. 
This graph contains a $K_{3,3,1,1}$ minor with partition $\{v_2, v_3, v_4\}, \{v_6, v_7, v_8\}, \{v_9\}, \{v_{10}\}$. 

\item[(iii)] If $v_6v_{10} \in E(G)$ (see Figure \ref{fig-10outerplanar1}(c)), then in $cG$ delete any edges incident to $v_1$ except $v_1v_7$, $v_1v_8$ and $v_1v_9$, then perform a $Y\nabla$-move at $v_1$ to create a graph containing the triangle $v_7v_8v_9$. 
Further, delete any edges incident to $v_6$ except $v_2v_6$, $v_3v_6$ and $v_4v_6$, then perform a $Y\nabla$-move at $v_6$ to create a graph containing the triangle $v_2v_3v_4$.  
Within this new graph, contract $v_5v_{10}$ to a new vertex $t$ to obtain a $K_7$ minor with vertices $\{v_2, v_3, v_4, v_7, v_8, v_9, t\}$.

\item[(iv)] If  $v_6v_{10} \notin E(G)$ and $v_6v_9\in E(G)$, (see Figure \ref{fig-10outerplanar1}(d)), then  
in $cG$ delete any edges incident to  $v_6$ except $v_6v_2$, $v_6v_3$ and $v_6v_4$, and perform a $Y\nabla$-move at $v_6$ to create a graph containing the triangle $v_2v_3v_4$. 
Within this new graph contract the edge $v_5v_9$ to a vertex $t$, and contract the edge $v_1v_7$ to a vertex $t_7$ to obtain a $K_7$ minor with vertices $\{v_2, v_3, v_4, t_7, v_8, v_{10}, t\}$.

\end{enumerate}

\begin{figure}[htpb!]
\begin{center}
\begin{picture}(530, 110)
\put(0,0){\includegraphics[width=6.6in]{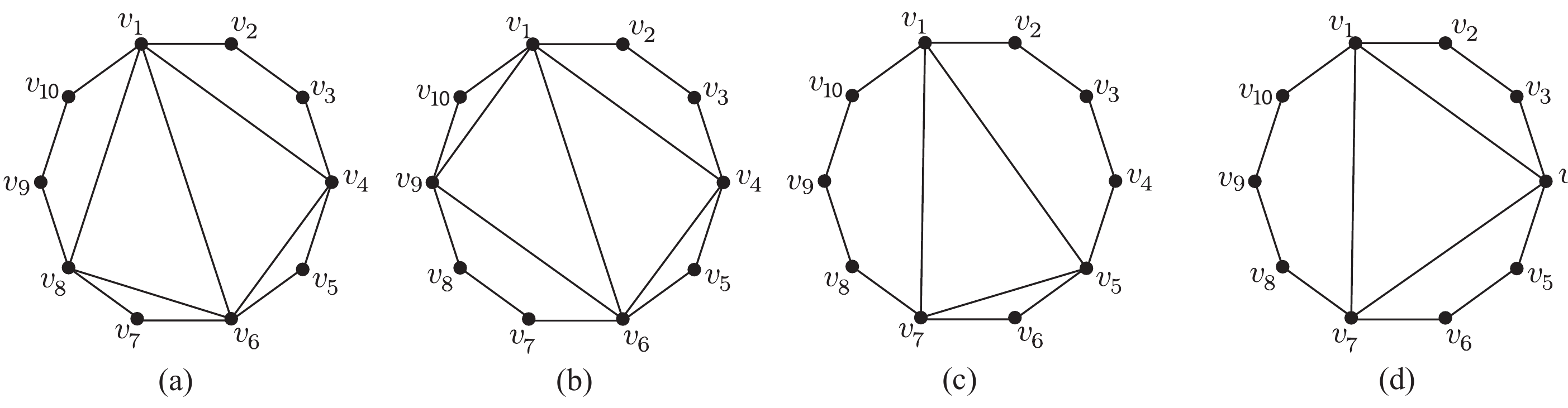}}
\label{fig-10outerplanar2}
\end{picture}
\caption{\small Outerplanar graphs with 10 vertices.}
\label{fig-10outerplanar2}
\end{center}
\end{figure}

\noindent {\bf{(a2)}} Assume there is no 4-chord of $\mathcal{C}$ incident to $v_1$ or $v_6$. There  are two  3-chords of  $\mathcal{C}$  incident to $v_1$ or $v_6$ and endpoints  in each $\mathcal{C}_1$ and $\mathcal{C}_2$. Assume $v_1v_4\in E(G)$.
\begin{enumerate}
\item[(i)] If $v_1v_8\in E(G)$ (see Figure \ref{fig-10outerplanar2}(a)), 
for any choice of edges which triangulate the quadrilaterals $v_1v_2v_3v_4$ and $v_8v_9v_{10}v_1$, the complement $cG$ contains as a subgraph the graph  Cousin 12 of $K_{3,3,1,1}$ described in \cite{GMN}.
This is a minor minimal IK graph with 9 vertices obtained from $K_{3,3,1,1}$ by two $\nabla Y$-moves followed by a  $Y\nabla$-move.
\item[(ii)] If  $v_6v_9\in E(G)$ (see Figure \ref{fig-10outerplanar2}(b)), 
obtain a $K_7$ minor of $cG$ by contracting the edges $v_1v_8$, $v_2v_6$ and $v_4v_9$.
\end{enumerate}

\noindent {\textbf{Case (b)}}  Assume the cycle $\mathcal{C}$ has no chord of length 5.
Then it has at least a chord of length 4. Assume $v_1v_7\in E(G) $. Up to symmetry, we recognize two cases.\\
\noindent {\bf{ (b1)}} If $v_1v_5\in E(G)$ (see Figure  \ref{fig-10outerplanar2}(c)), then 
the complement $cG$ contains the graph obtained through two $\nabla Y$-moves from $K_7$ with vertex set $\{v_2, v_3, v_4, v_6, v_8, v_9, v_{10}\}$: one  $\nabla Y$-move over the triangle $v_2v_3v_4$ with new vertex $v_7$ and one $\nabla Y$-move over the triangle $v_8v_9v_{10}$ with new vertex $v_5$. 

\noindent {\bf{(b2)}} If $v_1v_4, v_4v_7\in E(G)$ (see Figure  \ref{fig-10outerplanar2}(d)),  then  in $cG$ delete any edge incident to  $v_4$ except $v_4v_8$, $v_4v_9$ and $v_4v_{10}$, then perform a $Y\nabla$-move at $v_4$ to create a graph containing the triangle $v_8v_9v_{10}$.
Within this graph contract edges $v_1v_5$ to $t_5$ and $v_2v_7$ to $t_2$ obtain a $K_7$ with vertex set $\{t_2, v_3, t_5, v_6, v_8, v_9, v_{10}\}$.

 \end{proof}

\begin{lemma} If $G$ is a wheel  with ten vertices then $cG$ is intrinsically knotted.
\label{w10}
\end{lemma}
\begin{proof}
The graph  $E_9+ e$ is a minor minimal intrinsically knotted graph with 9 vertices  described in \cite{GMN}.
 The complement of  $E_9+ e $ in $K_{10}$  contains the 10-wheel as a subgraph. 
 See Figure \ref{fig-wheel10}.
Thus, the complement $cG$ contains $E_9+e$ as a subgraph and therefore  it is intrinsically knotted.
\end{proof}

\begin{figure}[htpb!]
\begin{center}
\begin{picture}(300, 120)
\put(0,0){\includegraphics[width=4in]{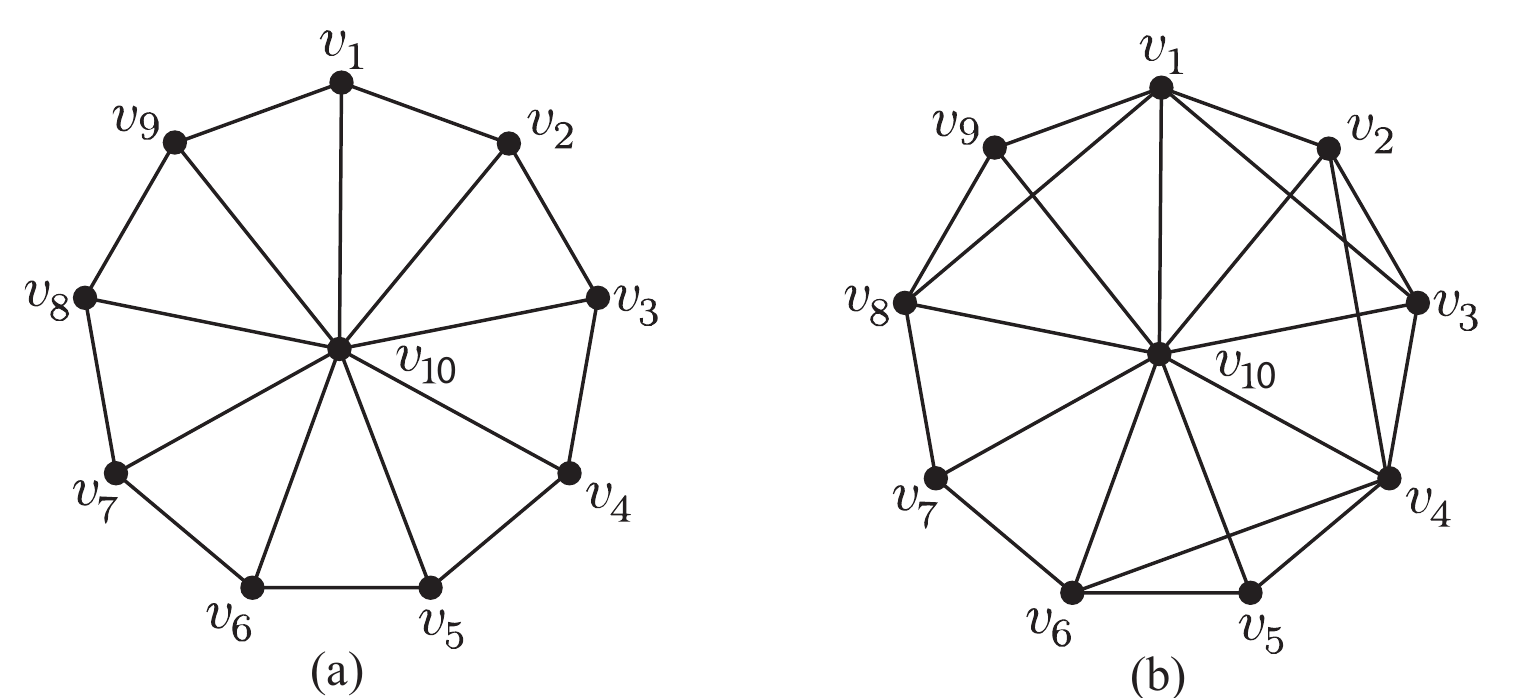}}
\end{picture}
\caption{\small (a) wheel graph with  ten vertices, (b) complement of $E_9+e$} in $K_{10}$.
\label{fig-wheel10}
\end{center}
\end{figure}

\begin{lemma} If $G$ is an elongated  triangular prism with ten vertices then $cG$ is intrinsically knotted.
\label{tp10}
\end{lemma}

\begin{proof}
An elongated  prism with ten vertices is obtained by subdividing the three non-triangular edges of the prism  with four vertices.  
These four vertices can be added in four different ways: (a) on three different edges; (b) on two edges with a 2-2 partition; (c) on two edges with a 3-1 partition; (d) all on one edge.
See Figure \ref{fig-prism10}.
In each case, we show that $cG$ contains a $K_{3,3,1,1}$ minor.

\begin{figure}[htpb!]
\begin{center}
\begin{picture}(480, 70)
\put(-10,0){\includegraphics[width=6.8in]{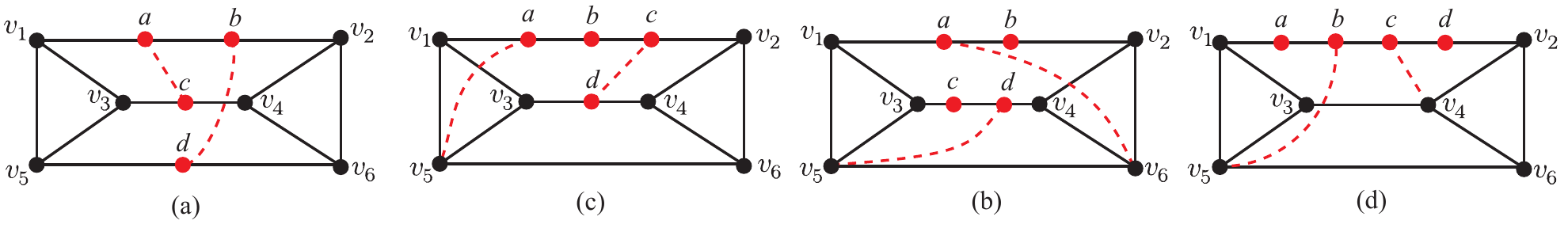}}
\end{picture}
\caption{\small Elongated prisms with ten vertices. Dashed edges are edges of the complement graph.}
\label{fig-prism10}
\end{center}
\end{figure}

\noindent {\textbf{Case (a)}} The four vertices are added  on three different edges of the elongated prism, as in Figure \ref{fig-prism10}(a). 
Within $cG$, contract edge $ac$ to vertex $t$ and  contract edge $bd$ to vertex $u$ to obtain a $K_{3,3,1,1}$ minor of $cG$ given by the partition $\{v_1,v_3, v_5\}, \{ v_2, v_4, v_6\}, \{t\}, \{u\}$.\\
 {\textbf{Case (b)} } The four vertices are added to two edges of the elongated prims with a 2-2 partition,  as in Figure \ref{fig-prism10}(b). 
Within $cG$, contract edge $dv_5$ to vertex $t_5$ and  contract edge $av_6$ to vertex $t_6$ to obtain a $K_{3,3,1,1}$ minor of $cG$ given by the partition $\{v_1,v_3, c\}, \{ v_2, v_4, b\}, \{t_5\}, \{t_6\}$.\\
 {\textbf{Case (c)} } The four vertices are added to two edges of the elongated prims with a 3-1 partition,  as in Figure \ref{fig-prism10}(c). 
Within $cG$, contract edge $av_5$ to vertex $t_5$ and contract edge $cd$ to vertex $t$ to obtain a $K_{3,3,1,1}$ minor of $cG$ given by the partition $\{v_1,v_3,t_5\}, \{ v_2, v_4, v_6 \}, \{b\}, \{t \}$.\\
 {\textbf{Case (d)}} The four vertices are added all one one edge of the elongated prism, as in Figure \ref{fig-prism10}(d). 
Within $cG$, contract the edge $bv_5$ to  vertex $t_5$   and  the edge $cv_4$ to vertex $t_4$ to obtain a $K_{3,3,1,1}$ minor of $cG$ given by the partition $\{v_1,v_3, a\}, \{ v_2, v_6, d \}, \{t_4\}, \{t_5\}$.

\end{proof}

\noindent Since for  $H$  a subgraph of $G$ of the same order,  $cG\subseteq cH$, 
Lemmas \ref{out10}, \ref{w10} and \ref{tp10}  give the following theorem.\\

\noindent \textbf{Theorem 3.} \textit{If  $G$ is a  non-separating planar graph on 10 vertices, then $cG$ is intrinsically knotted.}

\begin{corollary}
For $n\ge10$, the complement of a non-separating planar graph on $n$ vertices is IK.
\label{cor10}
\end{corollary}

\begin{remark}
The bound  $n\ge 10$ in Corollary \ref{cor10} is the best possible.
If $G$ is the 9-wheel, then $cG\setminus v$ =$K_8\setminus C_8$. 
Here $v$ is the isolated point within the complement of the wheel.
Since it has 20 edges, $K_8\setminus C_8$ is 2-apex and it is therefore knotlessly embeddable \cite{M}.
\end{remark}

\noindent \textbf{Aknowledgments.} The authors would like to thank Hooman Dehkordi, Graham Farr, and Ramin Naimi for  the helpful conversations.

\end{document}